\documentclass[psamsfonts]{amsart}

\usepackage{amssymb,amsfonts}
\usepackage[all,arc]{xy}
\usepackage{enumerate}
\usepackage{mathrsfs}
\usepackage{tikz}
\usepackage{graphicx} 
\usepackage{amsmath}
\usepackage{bbm}
\usepackage{pgf,tikz,pgfplots}
\usepackage[english]{babel}
\usepackage[utf8]{inputenc}
\usepackage[super]{nth}
\usepackage{amsmath}
\usepackage{soul}
\usepackage{mathrsfs}
\usetikzlibrary{arrows}
\newtheorem{thm}{Theorem}[section]

\newtheorem{prop}[thm]{Proposition}
\newtheorem{lem}[thm]{Lemma}
\newtheorem{conj}[thm]{Conjecture}

\theoremstyle{definition}
\newtheorem{defn}[thm]{Definition}

\newtheorem{obs}[thm]{Observation}

\newtheorem{col}[thm]{Corollary}
\newtheorem{rem}[thm]{Remark}

\newtheorem{app}[thm]{Application}
\def\p{{\mathfrak p}}

\makeatletter
\let\c@equation\c@thm
\makeatother
\numberwithin{equation}{section}

\title{Packing properties of cubic squarefree monomial ideals}
\author[A. Alilooee]{Ali Alilooee}
\address{Bradley University, Department of Mathematics and Statistics,
Bradley Hall,
Peoria, IL, USA }
\email{a20480m2018@gmail.com}
\author[A. Banerjee]{Arindam Banerjee}
\address{Ramakrishna Mission Vivekenanda University, Belur, West Bengal, India}
\email{123.arindam@gmail.com}

\begin{document}

\maketitle

\begin{abstract}
The symbolic powers, in general, are not equal
to the ordinary powers. Therefore, one interesting question here is for what classes of ideals ordinary and symbolic powers coincide? The answer to this question for squarefree monomial ideals may be packing property. In this paper, we classify all cubic path ideals for those the symbolic and ordinary powers coincide. 
\end{abstract}

\section{introduction}
\hspace{.1 in}  In this article we study the symbolic powers of path ideals. 
Let $I$ be an ideal in a Noetherian domain $R$, its $n^{\text{th}}$ symbolic power is defined as follow
$$I^{(n)}=\displaystyle\bigcap_{\p\in \text{Ass}_{R}(R/I)}(I^{n}R_{\p}\cap R).$$
Symbolic powers have been a historically important topic. Krull's famous proof of principal ideal theorem uses them. The proof of the Hartshorne-Lichtenbaum vanishing theorem also uses symbolic powers. In case of the monomial ideals the study of the symbolic power is intimately connected with the so called packing problem. In the context of linear optimization, Conforti and Cornuejols made a conjecture that relates to characterization of the set of square-free monomials ideals whose symbolic and ordinary powers are equal.
A hypergraph $\mathcal{H}$ satisfies the {\bf K{\"o}nig property} if the maximum number of independent hyperedges of $\mathcal{H}$ equals the height of the edge ideal $I(\mathcal{H})$. A hypergraph $\mathcal{H}$ is said to satisfy the {\bf  packing property} if all of the minors of $\mathcal{H}$ satisfy the K{\"o}nig property.  Conforti and Cornuéjols conjectured in \cite{conforti1994} that a hypergraph satisfies the max-flow min-cut property if and only if it satisfies the packing property (packing property can be rephrased in terms of ideals for more details see \cite{Dao2018}). By using Conforti and Cornuejols conjecture and Theorem 2.6 we can make the following conjecture. 
\begin{conj}\label{conjecture1994}Let $I=I(\mathcal{H})$ be the edge ideal of a hypergraph. The following conditions are equivalent. 
\begin{enumerate}
\item  $I^{(n)}=I^{n}$ for all $n\geq 1$.
\item $\mathcal{H}$ satisfies the packing property.

\end{enumerate}

\end{conj}
In the graph theory it is well-known (Cf. \cite{Villareal2009}, Proposition 4.27) that a graph $G$ satisfies the packing property if and only if $G$ is bipartite. Then it is straightforward to see that Conjecture~\ref{conjecture1994} for a squarefree ideal of degree 2 is correct. Many researchers have investigated this conjecture. In this paper we will study the packing property for cubic squarefree monomial ideals. Our main contribution in this paper is to solve this conjecture for all three path ideals. Our result about 3-path ideals is the following:\\
 
\begin{thm}
Let $G$ be a connected graph and $t\geq 2$ be an integer and let $J=I_3(G)$ be the cubic path ideal of $G$. Then $J^{(n)}=J^{n}$ for all $n\geq 1$ if and only if $G$ is a path graph $P_{k}$ or $G$ is the cycle $C_{3k}$ when $k=1,2,3$.   
\end{thm}

We organize the paper as follows.  In section (2) we review some necessary preliminaries. Then in section (3) we shall study cubic path ideals.

\section{Preliminaries} 

\begin{defn}
Let $X=\{x_1,x_2,\dots,x_n\}$ be a finite set. A {\bf simple hypergraph} on $X$ is a family $\mathcal{H}=(E_1,\dots,E_q)$ of subsets of $X$ such that 
\begin{enumerate}
\item All $E_i$s are nonempty;\\
\item $\displaystyle\cup_{i=1}^{q} E_i=X$;\\
\item None of the $E_i$s is contained within another.

\end{enumerate}
In this paper all hypergraphs are simple. The elements of $X$ are called {\bf vertices} and the sets $E_1,E_2,\dots,E_q$ are called the {\bf hyperedges} of $\mathcal{H}$. 
We sometimes denote the hypergraph $\mathcal{H}$ over the vertex set $X$ and with the edge set $\mathcal{E}=\{E_1,E_2,\dots,E_q\}$ with the pair $\mathcal{H}=(X,\mathcal{E})$.   
\end{defn}
\begin{defn}
A hypergraph $\mathcal{H}$ is called {\bf $r$-uniform} if for each $E\in \mathcal{H}$ we have $|E|=r$. It is obvious that a 2-uniform hypergraph is a graph.  
\end{defn}
\begin{defn}
Let $\mathcal{H}=(E_1,E_2,\dots,E_q)$ be a simple $r$-uniform hypergraph over the vertex set $X$ and $k$ be a field. We define the {\bf edge ideal} of $\mathcal{H}$ as a squarefree
monomial  ideal $I(\mathcal{H})$ in the polynomial ring $R=k[X]$ defined as
$$I(\mathcal{H})=(x_{i_1}\dots x_{i_r}:\{i_1\dots i_r\}\in \mathcal{H}).$$
\end{defn}

We recall the definition of Mengerian hypergraphs here. 
\begin{defn}\label{min-max}
Let $A$ be the edge-vertex incidence matrix of the hypergraph $\mathcal{H}$. Then $\mathcal{H}$ is called a {\bf Mengerian hypergraph} if for all $c\in \mathbb{N}^{n}$ we have   
 
$$\min\{c\cdot x|Ax\geq \mathbbm{1},x\in\mathbb{N}^{n}\}=\max\{y\cdot\mathbbm{1}|yA\leq c; y\in \mathbb{N}^{m}\}$$

\end{defn}
\begin{thm}[\cite{Villarreal2007}, \cite{Herzog2008}, \cite{Tai2019}, \cite{Trung2006}]\label{Tai}
Let $I$ be the edge ideal of a hypergraph $\mathcal{H}$. Then $I^{(n)}=I^{n}$ for all $n\geq 1$ if and only if $\mathcal{H}$ is Mengerian.
\end{thm}
\begin{prop}[\cite{Berge1989}, page 199, Proposition 1 and 2]
Let $I$ be a squarefree monomial ideal and $x$ be a variable and assume $I^{(n)}=I^{n}$ for all $n\geq 1$, then $(I:x)^{(n)}=(I:x)^{n}$.
 
\end{prop}
For sqaurefree monomial ideals we have a very useful following proposition. We will use this fact in Chapter 3 to show some of our results 
\begin{prop}[\cite{Villarreal2015}, Proposition 4.3.25]~\label{easyfact}
Let $R$ be a polynomial ring over a field $k$ and $I$ be a squarefee monomial ideal. If $\p_1,\p_2,\dots,\p_r$ are all the minimal primes of $I$, then for all $n\geq 1$ we have 
$$I^{(n)}=\p_1^{n}\cap\p_2^{n}\cap\dots\cap\p_r^{n}.$$  
\end{prop}
We also recall the following remark from \cite{Dao2018}. We will use this remark to prove one of our main result in Chapter 3. 
\begin{rem}[\cite{Dao2018}, Remark 4.12]\label{Craigmethod}
Let $J$ be a squarefree monomial ideal and let $x$ be a variable. We let $I_{x}=(J:x)$ and $I$ be an ideal generated by all monomials in $J$ which do not involve $x$. Put $I_x=I+L$ where $L$ is an ideal generated by all monomials in $I_x$ which are not in  $I$. Then if we assume that $I^{(n)}=I^{n}$ and 
$I_{x}^{(n)}=I_{x}^{n}$ for all $n\geq 1$, then we have $J^{(n)}=J^n$ for all $n\geq 1$ if and only if we have 
$$\displaystyle I^k\cap I^{i}L^{n-i}\subset \sum_{j=k}^{n} I^{j}L^{n-j}=I^{k}I_{x}^{n-k}$$
for all $k$ and $i$ in which we have $0\leq i<k\leq n$.  
\end{rem}

\section{Path ideals}

Path ideals first were introduced by Conca and De Negri in \cite{conca1999}. In this section we will study symbolic powers of a cubic path ideal of a graph. First we will recall the definition of the path ideal. 
\begin{defn}
Let $G=(V,E)$ be a simple graph. A sequence of distinct vertices $v_1,v_2,\dots,v_n$ is called a path if $\{v_i,v_{i+1}\}$ is an edge in $G$ for each $1\leq i<n$. The number of edges in a path is defined as the length of that path. We define the {\bf path ideal} of $G$, denoted by $I_t(G)$, to be the
ideal of $k[V]$ generated by the monomials of the form
$x_{i_1}x_{i_2}\dots x_{i_t}$ where $x_{i_1},x_{i_2},\dots,x_{i_t}$ is
a path in $G$. 
\end{defn}
\begin{defn}
Let $G=(V,E)$ be a simple graph and $I_t(G)$ be a path ideal of length $t-1$ when $2\leq t$. Then we define the {\bf path hypergraph} as a hypergraph whose vertices are $V$ and whose edges consist of all paths of the length $t-1$. This hypergraph is denoted by $\mathcal{H}_t(G)$.
\end{defn}
From Theorem~\ref{Tai}, [\cite{Adam2010}, Theorem 2.7] and [\cite{Herzog2008}, Theorem 3.2] we have the following result. 
\begin{col}\label{Adam'thm}
Let $t\geq 2$ be an integer and $G$ be a rooted tree (i.e. A rooted tree is a tree with one vertex chosen as a root). The path hypergraph  $\mathcal{H}_t(G)$ is a simplicial tree (for details about simplicial trees see \cite{Faridi2011}). Moreover for the path ideal $J=I_t(G)$ we have $J^{(n)}=J^{n}$ for all $n\geq 1$.  
\end{col}
Now we are ready to show when the $t$-path hypergraph of a cycle is $t$-partite.
Now we are ready to turn our attention to the class of cubic path ideals of a graph. To prove one of our main results on this paper (i.e. Theorem~\ref{maintheorem}) we need to prove the following theorem. 
\begin{thm}\label{theorem:meandArindam}
Let $G$ be a simple graph over $V=\{x_1,x_2,\dots,x_n\}$ and $\mathcal{H}_3(G)$ be a path hypergraph of $G$ of degree 3 and let $I=I_3(G)$ be the cubic path ideal of $G$.
Then we have the following 
\begin{enumerate}
\item If there is a vertex of degree 3 in $G$, then $I^{(2)}\neq I^2$.
\item If there is a copy of $C_{3k+1}$ or $C_{3k+2}$ as a subgraph in $G$ for a natural $k$, then 
$$I^{(k+1)}\neq I^{k+1}.$$ 
\end{enumerate}  
\end{thm}

\begin{proof}
We first assume there is a vertex, say $x_1$, of degree 3 in $G$. Then there are
vertices $x_2, x_3$ and $x_4$ in $G$ such that the edges $\{x_1, x_2\}, \{x_1, x_3\}, \{x_1, x_4\}$ are
in $G$. Then we have
$$x_1x_2x_3, x_1x_3x_4, x_1x_2x_4\in I.$$
We claim $x_1^2x_2x_3x_4\in I^{(2)}\backslash I^2$. To see this note that $\text{deg}(x_1^2x_2x_3x_4) = 5$ then it is obvious
that $x_1^2x_2x_3x_4\notin I^2$. The reason that $x_1^2x_2x_3x_4\in I^{(2)}$ comes from that fact that
$K =\left\langle  x_1x_2x_3, x_1x_3x_4, x_1x_2x_4\right\rangle$ is a subhypergraph of $\mathcal{H}_3(G)$ and each minimal
vertex cover of $\mathcal{H}_3(G)$ must contain two vertices of $\{x_2 x_3, x_4\}$ or just $x_1$. Since each associated prime of the ideal $I$ is given by a vertex cover of $\mathcal{H}$ we have $x_1^2x_2x_3x_4$ belongs to the second power of the all associated primes of $I$. Then from Proposition~\ref{easyfact} we can conclude that our claim is true.  

Now we assume there is a copy of subgraphs $C_{3k+1}$ or $C_{3k+2}$ in $G$. Without loss of
generality we can assume the vertex set of $C_{3k+1}$ or $C_{3k+2}$ is $\{x_1,\dots, x_{3k+1}\}$ or
$\{x_1,\dots, x_{3k+2}\}$. We have
\begin{align*}
&x_1x_2x_3, x_2x_3x_4,\dots, x_{3k-1}x_{3k}x_{3k+1}, x_1x_{3k}x_{3k+1}, x_1x_2x_{3k+1}\in I&          \\
 &\text{or}& \\
&x_1x_2x_3, x_2x_3x_4,\dots, x_{3k}x_{3k+1}x_{3k+2}, x_1x_{3k+1}x_{3k+2}, x_1x_2x_{3k+2}\in I.          & 
\end{align*}
First note that from Proposition 4.14 in [13] we have
$$\text{ht}(I_3(C_{3k+1})) = \text{ht}(I_3(C_{3k+2})) = k + 1.$$
Since $G$ has a copy of $C_{3k+1}$ or $C_{3k+2}$ and then $\mathcal{H}_3(C_{3k+1})$ or $\mathcal{H}_3(C_{3k+2})$
is a subhypergraph of $\mathcal{H}_3(G)$ and therefore, every minimal vertex cover of $\mathcal{H}_3(G)$
must contain at least $k + 1$ of the vertices of $C_{3k+1}$ or $C_{3k+2}$. Then from Proposition~\ref{easyfact} we have
\begin{align*}
f=x_1x_2x_3x_4x_5\dots x_{3k}x_{3k+1}\in I^{(k+1)}&        
 \text{or}& g=x_1x_2x_3x_4x_5\dots x_{3k+1}x_{3k+2}\in I^{(k+1)}.  
\end{align*}

On the other hand, since we have $\text{deg}(f)=3k+1$ and $\text{deg}(g)=3k+2$ clearly we can conclude that   
\begin{align*}
&f,g\notin I^{k+1}.
\end{align*}

\end{proof}
Now we are ready to prove one of our main results, Theorem~\ref{maintheorem}. We will use the similar technique as in \cite{Dao2018}. First we need the following lemma for graphs without any vertex of degree $>2$ and with no cycle $C_{3k\pm 1}$ for some integer $k$. Note that this class is equivalent to  graphs for those each connected components are either $C_{3k}$ or a path.  This result is a counterpart of Lemma 4.15 in \cite{Dao2018} for the case of the edge ideal of a graph.  
\begin{lem}\label{Ali'sLemma}
Let $G$ be a simple graph  such that all vertices have degree at most two and suppose $G$ has no cycle of 
 length $3k\pm 1$ for some $k$. Also suppose 
$p_1,\dots,p_{e+1}$ are paths of length two such that 
$\displaystyle\prod_{i=1}^{e+1} p_i$ divides $\displaystyle\left(\prod_{i=1}^{e}p'_i\right)\left(\prod_{z\in D}z\right)$ where the $p'_i$s are paths 
of length two and $D$ is a set of vertices containing no edges. 
Then there are $z_1,z_2\in D$ such that there is a path of length two in $G$ which connects $z_1$ and $z_2$.    
\end{lem}
\begin{proof}
We show that we must have two paths $p_1$ and $p_2$ such that at least one of them contains two elements of $D$. The reason for this claim is that the degree of the product of the $p_i$s is $3e+3$ and the degree of the product of the $p'_i$s is $3e$ then since there is no edge on the elements of $D$, we must have at least two paths say $p_1$ and $p_2$ such that $p_1$ contains two elements of $D$ and $p_2$ contains one vertex of $D$.

Note that it is not possible to have three paths of the $p_i$s such that each has just one vertex in $D$. Because if we assume there are three paths of the $p_i$s say $p_1,p_2,p_3$ such that there are $z_i\in D$ for $i=1,2,3$ in which $p_i$ contains $z_i$ and the other vertices of $p_i$ are not in $D$, then the degree of the product of the rest is $3e-6$. Therefore the product of the paths $p_4,p_5,\dots,p_{e+1}$ can be covered by the product of $e-2$ paths of the $p'_i$s. We have two more paths $p_1',p_2'$. Now if we assume that the remaining vertices of $p_1p_2p_3$ which are not in $D$ all are covered by the product of $p_1'p_2'$, then we must have a cycle of length 7 in $G$ or we must have a vertex of degree $>$ 2 in $G$ and both are contradictions. Then one of $p_1,p_2$ or $p_3$ must contain two vertices of $D$. 

We write $p_1=z_1z_2a$ and $p_2=z_3bc$ where $z_1,z_2,z_3\in D$. Note that since there is no edge on the vertices in $D$ we can say $a\notin D$ and since $z_1,z_2\in D$ we can conclude that $\{z_1,z_2\}$ is not an edge in $G$. Therefore we can conclude that $z_1$ is connected to $z_2$ by $p_1$ which is a path of length two.

\end{proof}
 In order to prove our main theorem of this chapter we need the following observation. 
\begin{obs}\label{niceObservation}
Let $G$ be the cycle $C_{3k}$ for a positive integer $k$ and let $J=I_3(C_{3k})$ be the cubic path ideal on $C_{3k}$. If $x$ is a variable of $C_{3k}$, we have the following:
\begin{enumerate}
\item If $k>3$, then $(J:x)^{(k+1)}\neq (J:x)^{k+1}$. Moreover $\mathcal{H}_{3}(C_{3k})$ is not Mengerian. 
\item If $k\leq 3$, then $(J:x)^{(n)}=(J:x)^n$ for all $n\geq 1$.
\end{enumerate}
\begin{proof}
Without loss of generality we can assume if we denote the vertex set of $G$ by $V=\{x_1,\dots,x_{3k}\}$, then $x=x_1$. 
Thus we have 
$$(J:x_1)=(I_3(P),x_3x_2,x_2x_{3k},x_{3k}x_{3k-1})$$
where $P$ is a path on the set $\{x_3,x_4,\dots,x_{3k-1}\}$. 

Since $(J:x_1)$ is a squarefree monomial ideal we can consider a hypergraph $\mathcal{H}$ such that the edge ideal of $\mathcal{H}$ is  $(J:x_1)$. Note that if $k=1,2,3$ it is straightforward to see that $\mathcal{H}$ is balanced (for details about balanced hypergraphs see \cite{Berge1989}) and then from Theorem 2.5 in \cite{Herzog2008} we have $\mathcal{H}$ is Mengerian and by using Theorem~\ref{Tai} we can conclude that the claim is true. 

Then we suppose $k>3$ and we set $f=x_2x_3x_4x_5^2x_6x_7x_8x_9^2x_{10}\dots x_{3k-1}x_{3k}$. Note that from Proposition 4.8 in \cite{Kubitzke2014} we can conclude that $\text{ht}(I_3(P))=k-1$. Therefore, we can say there is a minimal vertex cover (say $C$) of the size $k-1$ in $\mathcal{H}_3(P)$. If $x_3\in C$ or $x_{3k-1}\in C$, then $C\cup\{x_2\}$ or $C\cup\{x_{3k}\}$ is a minimal vertex cover of $\mathcal{H}$ and otherwise $C\cup C'$ is a minimal vertex cover of $\mathcal{H}$ where $C'$ is a minimal vertex cover of the path $x_3x_2,x_2x_{3k},x_{3k}x_{3k-1}$. It is obvious that $|C'|=2$ and so we can conclude that the minimum size of minimal vertex covers in $\mathcal{H}$ is $k$.     
Let $C_0$ be a vertex cover of $\mathcal{H}$, then we have 
\begin{align}\label{formula}
|C_0|=k \Leftrightarrow \{x_2,x_5,x_8,x_{11},\dots,x_{3k-1}\}\subset C_0 &\text{\hspace{.02 in} or} &\{x_3,x_6,x_9,\dots, x_{3k}\}\subset C_0. 
\end{align}
Then we can write $f\in I^{(k+1)}$ where $I=(J:x_1)$. To see this note that each vertex cover has at least $k$ elements. If we choose a vertex cover $C_0$ of length $>k$, it is clear that $f\in {I(C_0)}^{k+1}$ where $I(C_0)$ is the prime ideal associated to $C_0$. On the other hand, from (\ref{formula}) we have all vertex cover of $\mathcal{H}$ of length $k$ has $x_5$ or $x_9$ and then it is clear that 
$f\in I(C_0)^{k+1}$. 

Now note that $f\notin I^{k+1}$. We will show this by the way of contradiction. If we assume $f\in I^{k+1}$ since we have $deg(f)=3k+1$ and  since powers on $x_2,x_3,x_{3k-1}$ and $x_{3k}$ are one then $f$ must be divided by $(e_1e_2)g$ where 
$$e_1,e_2\in\{x_2x_3,x_2x_{3k},x_{3k-1}x_{3k}\}$$ 
and $g\in I^{k-1}_3(Q)$ and $Q$ is a path on the vertices $\{x_4,x_5,\dots,x_{3k-2}\}$. On the other hand since $I^{k-1}_3(Q)\subset I^{(k-1)}_3(Q)$ we  can conclude that $g\in I^{(k-1)}_3(Q)$. Note that since powers on $x_2,x_3,x_{3k-1}$ and $x_{3k}$ are one then $e_1,e_2\in\{x_2x_3,x_{3k-1}x_{3k}\}$. Then we have 
$$f/e_1e_2=x_4x_5^2x_6x_7x_8x_9^2\dots x_{3k-2}\in I^{(k-1)}_3(Q)$$
and it is a contradiction because $C_{0}'=\{x_6,x_8,x_{11},\dots,x_{3k-4}\}$ is a vertex cover of the length $k-2$ and 
$f/e_1e_2\notin (x_6,x_8,x_{11},\dots,x_{3k-4})^{k-1}$.

\end{proof}  
\end{obs}

To prove the main theorem of this chapter we need the following proposition. We let $J=I_{3}(G)$ be the cubic path ideal of the graph $G$ and $I$ be an ideal generated by all monomials in $J$ which do not involve $x$ (when $x$ is a vertex of $G$). Write $(J:x)=I+L$ when $L$ is an ideal generated by monomials in $(J : x)$ which are not in $I$. Then we have the following result.  
\begin{prop}\label{newProp}
Let $G$ be a graph such that all vertices have degree at most two and assume $G$ has no cycle of length $3m\pm 1$ for some  
$m$, then we have 
$$\displaystyle I^k\cap I^{i}L^{n-i}\subset \sum_{j=k}^{n} I^{j}L^{n-j}$$
for all integers $k$ and $n$ in which we have $i<k\leq n$.

\end{prop}
\begin{proof}
We use induction on $n$ and a backward induction on $k$. First note that for $k=n$ the assertion is obvious then we can assume that $k<n$. 
We assume that there is a monomial $f\in I^k\cap I^{i}L^{n-i}$ such that $f\notin \sum_{j=k}^{n} I^{j}L^{n-j}$. Therefore, $f$ can be written 
in the following forms
\begin{equation}\label{equation(1)}
f=\displaystyle\left(\prod_{1\leq j\leq a}a_jb_jc_j\right)\left(\prod_{1\leq j\leq b}u_jv_jg_j\right)\left(\prod_{1\leq j\leq c}m_j\right)
\left(\prod_{y\in F}y\right)
\end{equation} 
where $a+b=k$, $m_j \in L$ and $a_jb_jc_j$ and $u_jv_jg_j$ belong to  $I$ for all $j$. Also we have 
 for each $1\leq j\leq a$ the product of exactly one pair of $a_j$, $b_j$ and $c_j$ is in $L$
and for each $1\leq j\leq b$ there are no pairs of $u_j, v_j$ and $g_j$ such that their product is in $L$. 
(Note that it's not possible to have three vertices such that products of two pairs of those are in $L$ because otherwise we have two paths which are incident $x$ like $abx$ and $acx$ and then we have a vertex of degree 3 or we have a square in $G$ which both are contradictions.) 
In this expression we may have some variables which products of no pairs of them is in $L$. These are listed in the set $F$. 

Since $f\in I^{i}L^{n-i}$ we also can have the following expression for $f$. 
\begin{equation}\label{equation(2)}
f=\displaystyle\left(\prod_{1\leq j\leq i}z_jn_jy_j\right)\left(\prod_{1\leq j\leq \ell}g'_je_j\right)
\left(\prod_{\omega\in W}w\right)
\end{equation} 
 where $z_jn_jy_j\in I$ and $g'_je_j\in L$ and $\ell\geq n-i$. Note that the $\omega$s are variables such that the products of no pairs is in 
$L$. 

If there is $z_jn_jy_j$ in (\ref{equation(2)}) such that it is in (\ref{equation(1)}) then we have 
$$\displaystyle \frac{f}{z_jn_jy_j}\in I^{i-1}L^{n-i}\cap I^{k-1}\subset \sum_{j=k-1}^{n-1} I^{j}L^{n-1-j}.$$
Note that the above inclusion can be concluded from the induction hypothesis. So we have $f\in \sum_{j=k}^{n} I^{j}L^{n-j}$ which is a contradiction. Therefore, we must not have a cubic monomial in common between (\ref{equation(1)}) and (\ref{equation(2)}) which are in $I$. 

We will show $F=\emptyset$. Suppose it is not true. Then for each $y\in F$ we can have two scenarios. It is possible to have $y\in W$ or we may have a monomial $z_jn_jy_j$ in (\ref{equation(2)}) which is divided by $y$. If $y\in W$, then we can cancel it from both sides of the equality. 
If we assume there is $z_jn_jy_j$ in (\ref{equation(2)}) which is divided by $y$, then we will rearrange $f$ in (\ref{equation(1)}) in a way that $z_jn_jy_j$ appears in the collection of cubic monomials in (\ref{equation(1)}). 

Also we can claim that $c=0$ in (\ref{equation(1)}). Suppose there is $m_j$ for $1\leq j \leq c$. Note that if $m_j$ divides 
$\prod_{1\leq j\leq \ell}g'_je_j$ then we can cancel $m_j$ from both sides of the equality and from the induction hypothesis for $n$ we 
have
$$\displaystyle \frac{f}{m_j}\in I^{k}\cap I^{i}L^{n-1-i}\subset \sum_{j=k}^{n-1} I^{j}L^{n-1-j}$$
(note that since $k< n$ we have $i<k\leq n-1$) and then we have $f\in \sum_{j=k}^{n} I^{j}L^{n-j}$ and it is a contradiction. 
If $m_j$ divides $\prod_{1\leq j\leq i}z_jn_jy_j$ then there is a variable $s$ such that $m_js=zns$ and $zn\in L$ and $zns\in I$. Since $s$
also belongs (\ref{equation(1)}) we can cancel $m_js$ from both sides of equality and then from the induction hypothesis on $n-1$ we have                                                                                                                               
$$\displaystyle \frac{f}{m_js}\in I^{k-1}\cap I^{i-1}L^{n-i}\subset \sum_{j=k-1}^{n-1} I^{j}L^{n-1-j}$$ 
and then $f\in \sum_{j=k}^{n} I^{j}L^{n-j}$ which is a contradiction. Then $c=0$.

Therefore, we have   $\displaystyle\left(\prod_{1\leq j\leq i}z_jn_jy_j\right)\left(\prod_{1\leq j\leq \ell}g'_je_j\right)$ divides 
$\displaystyle\left(\prod_{1\leq j\leq a}a_jb_jc_j\right)\left(\prod_{1\leq j\leq b}u_jv_jg_j\right)$ and then 
$\displaystyle\left(\prod_{1\leq j\leq i}z_jn_jy_j\right)$ divides 
$$\displaystyle c_1\dots c_{\ell}\left(\prod_{\ell+1\leq j\leq a}a_jb_jc_j\right)\left(\prod_{1\leq j\leq b}u_jv_jg_j\right).$$
Note that there is no edge between each pair of the $c_j$s. Because otherwise we must have a cycle of 
length 7 in $G$ which is a contradiction. Also we have 
$(a-\ell)+b=k-\ell\leq k-n+i<i$, then the number of monomials that we have in $\displaystyle\left(\prod_{1\leq j\leq i}z_jn_jy_j\right)$ is greater than the number of monomials in 
$\displaystyle\left(\prod_{\ell+1\leq j\leq a}a_jb_jc_j\right)\left(\prod_{1\leq j\leq b}u_jv_jg_j\right).$

By using Lemma~\ref{Ali'sLemma} we can conclude that there are $c_{i_1},c_{i_2}$ where $i_1,i_2\in\{1,2,\dots,\ell\}$ such that there is a path of two connecting $c_{i_1}$ and $c_{i_2}$. On the other hand, note that $c_{i_1},c_{i_2}$ are vertices belonging to the paths of form $m_1n_1c_{i_1}$ and $m_2n_2c_{i_2}$ where $m_1n_1$ and $m_2n_2$ are in $L$.  By the definition of $L$ we can conclude that there are paths $xm_1n_1$ and $xm_2n_2$ in $G$. Therefore, we can write there are paths of length  three between $x$ and $c_{i_1}$, and $x$ and $c_{i_2}$. Since we know there is no vertices of degree $\geq 3$ we have there is a cycle of length 8 in $G$. It is a contradiction and then we can conclude  that our claim is true.     
\end{proof}

\begin{col}\label{newCol}
Let $G$ be the cycle $C_{3k}$ for a positive integer $k$ and let $J=I_3(C_{3k})$ be the cubic path ideal on $C_{3k}$. Assume $k\leq 3$, then $J^{(n)}=J^{n}$ for all $n\geq 1$. 

\end{col}
\begin{proof}
We use Remark~\ref{Craigmethod}. We pick a variable $x$ in $C_{3k}$ and we define $I_{x}=(J:x)$. Let $I$ be an ideal of monomials in $J$ which do not involve $x$. Write $I_x=I+L$ where $L$ is an ideal in $I_x$ which are not in $I$. It is straightforward to see that $I$ is the cubic path ideal over the path graph over the vertex set $V(C_{3k})\backslash\{x\}$. Thus from Corollary~\ref{Adam'thm} we have 
\begin{align}
I^{(n)}=I^{n}\hspace{.3 in}\text{for $n\geq 1$.}
\end{align}
Also it is clear that $L$ is the edge ideal of a subtree of $G$. From Observation~\ref{niceObservation} we have  $I_{x}^{(n)}=I_{x}^{n}$ for all $n\geq 1$. Therefore, Remark~\ref{Craigmethod} and Proposition~\ref{newProp} settle the claim.

\end{proof}

\begin{thm}\label{maintheorem}
Let $G$ be a connected graph and let $J=I_3(G)$ be the cubic path ideal of $G$. Then $J^{(n)}=J^{n}$ for all $n\geq 1$ if and only if $G$ is a path graph $P_{k}$ or $G$ is the cycle $C_{3k}$ when $k=1,2,3$.   
\end{thm}
\begin{proof}
If $G$ is the path graph $P_k$ or the cycle $C_{3k}$ for $k=1,2,3$ from Corollary~\ref{newCol} and Corollary~\ref{Adam'thm} we can write $J^{(n)}=J^{n}$ for all $n\geq 1$. So we assume for all $n$ we have $J^{(n)}=J^{n}$. Then from Theorem~\ref{theorem:meandArindam} part (1) we can conclude that $G$ has no vertex of the length 3. Then we can conclude that $G$ must be a path graph or a cycle. On the other hand from Theorem~\ref{theorem:meandArindam} and Observation~\ref{niceObservation} we can say $G$ is a path graph of $C_{3k}$ for $k=1,2,3$. 

\end{proof}
\section{Applications}

 Our results on path ideals in Theorem~\ref{maintheorem} give the following applications in the linear programming.  
\begin{app}\label{application1} 
Let $\mathbf{M}$ be a square matrix of order $p$ defined as follow and ${\bf a}\in\mathbb{N}^{p}$.  
 \[\mathbf{M}=\begin{bmatrix}
    1 & 0 & 0 & 0 &\dots  &  0 & 0 & 1 & 1\\
    1 & 1 & 0 & 0 & \dots  & 0 & 0  & 0 & 1\\
		1 & 1 & 1 & 0 & \dots  & 0 & 0 & 0  & 0\\
		0 & 1 & 1 & 1 &\dots  & 0 & 0 & 0 & 0\\
\vdots & \vdots & \vdots & \vdots & \ddots & \vdots&\vdots&\vdots&\vdots \\
		0 &0 & 0 & 0  & \dots  & 1& 1 & 0 & 0\\
    0 &0 & 0 & 0  & \dots  & 1& 1 & 1 & 0\\
		0 & 0 & 0 & 0 & \dots  & 0 &1 & 1 & 1
\end{bmatrix}
\]
 Consider the following linear programming problems:\\
\begin{enumerate}
\item $\begin{array}{ll@{}ll}
\text{maximize}  &  \mathbbm{1}^{p}\cdot \mathbf{y}, &\\
\text{subject to}& \mathbf{M}\cdot\mathbf{y}\leq {\mathbf a}, {\mathbf y}\in   \mathbb{N}^{p} &\\
       \end{array}$
 
\item 
$\begin{array}{ll@{}ll}\label{linearProg}
\text{minimize}  & \mathbf{a}\cdot \mathbf{z},  &\\
\text{subject to}& \mathbf{M}^{\top}\cdot\mathbf{z}\geq \mathbbm{1}^{p}, \mathbf{z}\in   \mathbb{N}^{p} &\\
       \end{array}$

\end{enumerate}
Now we denote the optimal values for theses linear programming problems by $\nu_{\mathbf{a}}(\mathbf{M})$ and 
$\tau_{\mathbf{a}}(\mathbf{M})$ respectively. Then we have $\nu_{\mathbf{a}}(\mathbf{M})=\tau_{\mathbf{a}}(\mathbf{M})$ if and 
only if $p=3k$ for $k=1,2,3$.
\end{app} 
 \begin{app}
 Let $\mathbf{M}$ be a square matrix of order $p\times(p+1)$ defined as follow and ${\bf a}\in\mathbb{N}^{p}$.  
 \[\mathbf{M}=\begin{bmatrix}
    1 & 0 & 0  &\dots  &  0 & 0\\
    1 & 1 & 0  & \dots  & 0 & 0\\
		1 & 1 & 1  & \dots  & 0 & 0 \\
		0 & 1 & 1  &\dots  & 0 & 0 \\
\vdots & \vdots & \vdots & \ddots &\vdots&\vdots \\
		0 &0 & 0   & \dots  & 1& 1  \\
    0 &0 & 0   & \dots  & 1& 1  \\
		0 & 0 & 0 & \dots  & 0 &1  
\end{bmatrix}
\]
 Now if we consider the linear programming problem (\ref{linearProg}) and we define $\nu_{\mathbf{a}}(\mathbf{M})$ and $\tau_{\mathbf{a}}(\mathbf{M})$ as the optimal values for theses linear programming problems then we have $\nu_{\mathbf{a}}(\mathbf{M})=\tau_{\mathbf{a}}(\mathbf{M})$ for all $p$. 
 \end{app}

\section*{Acknowledgment}

We are very thankful to Professor R. Villarreal and Professor D. Ullman for their helpful comments and suggestions.  We gratefully acknowledge the helpful computer algebra system Macaulay2~\cite{M2} , without which our work would have been difficult or impossible. 

\bibliographystyle{plain}
\bibliography{mine1}

\end{document}